\documentclass[hidelinks,11pt]{amsart}

\usepackage{amsmath,amssymb,amsfonts,hyperref,graphicx,tikz,mathrsfs,amsaddr,blindtext}

\usepackage[labelsep=period, font=footnotesize, labelfont=bf]{caption}

\makeatletter
\@namedef{subjclassname@2020}{2020 Mathematics Subject Classification}
\makeatother

\usepackage[margin=1in,footskip=0.25in]{geometry}

\theoremstyle{plain}
\newtheorem{theorem}{Theorem}[section]
\newtheorem{lemma}[theorem]{Lemma}
\newtheorem{corollary}[theorem]{Corollary}

\theoremstyle{definition}

\newtheorem{remark}[theorem]{Remark}

\numberwithin{equation}{section}

\newcommand{\bmi}[1]{\mbox{\boldmath $ #1$}}

\newsavebox\ideabox
\newenvironment{idea}
{\begin{equation}
\begin{lrbox}{\ideabox}
\begin{minipage}{\dimexpr\columnwidth-2\leftmargini}
\setlength{\leftmargini}{0pt}
\begin{quote}}
{\end{quote}
\end{minipage}
\end{lrbox}\makebox[0pt]{\usebox{\ideabox}}
\end{equation}}

\begin{document}

\title[]{On the location of  zeros of  the   Laplacian matching polynomials of graphs}

\author[]{Jiang-Chao Wan, Yi Wang, Ali Mohammadian}
\address{\rm{\normalsize{School of Mathematical Sciences, Anhui University, Hefei 230601, Anhui, China}}}
\thanks{{Email adress:}wanjc@stu.ahu.edu.cn (J.-C. Wan),  wangy@ahu.edu.cn (Y. Wang, corresponding author), ali\_m@ahu.edu.cn (A. Mohammadian)}
\thanks{{\it Funding.} The research of the second author is supported by the National Natural Science Foundation of China with  grant numbers 11771016 and  11871073. The research of the third author is supported by the Natural Science Foundation of Anhui Province  with  grant number 2008085MA03.}



\sloppy

\maketitle

\begin{abstract}
The   Laplacian matching polynomial of a graph $G$, denoted by $\mathscr{L\hspace{-0.7mm}M}(G,x)$, is a new graph polynomial whose all roots  are nonnegative real numbers.
In this paper,  we  investigate  the location of zeros of  the   Laplacian matching polynomials.
Let $G$ be a connected graph.
We show    that $0$ is a root of   $\mathscr{L\hspace{-0.7mm}M}(G, x)$ if and only if $G$ is a tree. We prove  that  the number of distinct positive zeros of $\mathscr{L\hspace{-0.7mm}M}(G,x)$ is at least equal to the length of the longest path in $G$.
It is also established   that the zeros of $\mathscr{L\hspace{-0.7mm}M}(G,x)$ and $\mathscr{L\hspace{-0.7mm}M}(G-e,x)$  interlace for each  edge $e$ of $G$.
Using the path-tree of $G$,
we present a linear algebraic approach to investigate  the largest zero of $\mathscr{L\hspace{-0.7mm}M}(G,x)$
and particularly to give   tight  upper and     lower   bounds   on it.
\end{abstract}

\section{Introduction}\label{sec1}

The graph polynomials, such as the characteristic polynomial, the chromatic polynomial, the independence polynomial, the matching polynomial,  and many others, are widely studied and play   important roles  in applications of graphs   in several diverse fields.
The  location of  zeros of  graph polynomials is a main  topic in algebraic combinatorics and can be used to describe some structures and parameters of graphs.
In this paper,  we focus on the location of zeros of  the Laplacian matching polynomials of graphs.
For more results on the location of zeros  of graph polynomials,  we refer to   \cite{Makowskya}.

Throughout this paper, all graphs are assumed to be finite, undirected, and without loops or multiple edges.
Let $G$ be a graph.
We    denote the vertex set of $G$   by $V(G)$ and the edge set of $G$ by  $E(G)$.
Let $M$ be a subset of $E(G)$. We  denote by  $V(M)$ the set of vertices of $G$ each of which is an endpoint of one of the edges in $M$.
If no two distinct edges in $M$ share a common endpoint, then $M$ is called a  {\it   matching} of $G$.
The set of matchings of $G$ is denoted by $\mathcal{M}(G)$.
A matching $M\in\mathcal{M}(G)$   is said to be {\it  perfect} if    $V(M)=V(G)$.
The {\it   matching polynomial} of $G$ is
$$\mathscr{M}(G,x)=\sum_{M\in\mathcal{M}(G)}(-1)^{|M|}x^{|V(G)\setminus V(M)|}$$
which was     formally defined  by Heilmann and Lieb   \cite{Heilmann}   in studying statistical physics, although it has  appeared independently in several different
contexts.

The matching polynomial is a fascinating mathematical object    and attracts considerable attention of researchers.
For an instance, by studying the multiplicity of zeros of the matching polynomials,
Chen and Ku  \cite{Ku} gave a   generalization of  the   Gallai--Edmonds   theorem which   is a structure theorem in classical     graph theory.
For another  instance,  using  a well known  upper bound on   zeros of the matching polynomials,
Marcus,   Spielman, and  Srivastava \cite{Marcus}    established that    infinitely  many   bipartite   Ramanujan graphs exist.
Some earlier facts  on the matching polynomials  can be found     in \cite{Godsil}.

We want to  summarize here  some  basic  features    of  the zeros  of the matching polynomial.
For this,  let us first  introduce some more notations and terminology which we need.
For a vertex $v$ of a graph  $G$,
we  denote by  $N_G(v)$   the set of all vertices of $G$   adjacent to $v$. The {\it   degree} of $v$ is defined as $|N_G(v)|$ and is   denoted by $d_G(v)$.
The maximum degree  and the minimum degree of the vertices of $G$ are  denoted by $\mathnormal{\Delta}(G)$ and $\delta(G)$, respectively.
For a subset $W$ of $V(G)$, we shall use $G[W]$ to denote the induced subgraph of $G$ induced by $W$  and we simply use
$G-W$ instead of  $G[V(G)\setminus W]$.
Also, for  a vertex $v$ of $G$, we simply  write $G-v$ for $G-\{v\}$.
For an edge  $e$ of $G$, we  denote  by  $G-e$   the  subgraph of $G$  obtained  by deleting the edge $e$.

Let $\alpha_1\leq\cdots\leq\alpha_n$ and $\beta_1\leq\cdots\leq\beta_m$
be respectively the zeros   of two real rooted polynomials $f$ and $g$ with $\deg f=n$ and $\deg g=m$.
We say  that the zeros of $f$ and $g$  {\it   interlace} if   either
$$\alpha_1\leq\beta_1\leq\alpha_2\leq\beta_2\leq\cdots$$
or
$$\beta_1\leq\alpha_1\leq\beta_2\leq\alpha_2\leq\cdots$$
in which case one clearly must have $|n-m|\leq1$.
We adopt the convention that the zeros of any  polynomial  of degree   $0$   interlace the zeros of  any other  polynomial.

For   any  connected graph $G$, the assertions given in  \eqref{thmmatching1}--\eqref{thmmatching2} are  known.
\begin{idea}\label{thmmatching1}
All the  roots  of $\mathscr{M}(G,x)$ are   real. Moreover, if $\mathnormal{\Delta}(G)\geq 2$, then the zeros of $\mathscr{M}(G,x)$ lie in the interval $(-2\sqrt{\mathnormal{\Delta}(G)-1}, 2\sqrt{\mathnormal{\Delta}(G)-1})$   \cite{Heilmann}.
\end{idea}
\begin{idea}\label{thmmatching3}
The number of distinct roots  of $\mathscr{M}(G,x)$ is at least equal to $\ell(G)+1$, where $\ell(G)$ is  the length of the longest path in $G$ \cite{Godsil2}.
\end{idea}
\begin{idea}\label{thmmatching2}
For each  vertex $v\in V(G)$, the zeros of $\mathscr{M}(G-v,x)$ interlace the zeros of  $\mathscr{M}(G,x)$. In addition,    the largest zero of $\mathscr{M}(G,x)$ has the multiplicity $1$ and is greater than the largest zero of $\mathscr{M}(G-v, x)$  \cite{Godsil1}.
\end{idea}

Recently, Mohammadian \cite{Ali} introduced a new graph polynomial  that is    called the {\it   Laplacian matching polynomial} and  is defined for  a  graph $G$ as
\begin{equation}\label{defali}\mathscr{L\hspace{-0.7mm}M}(G, x)=\sum_{M\in\mathcal{M}(G)}(-1)^{|M|}\left(\prod_{v\in V(G)\setminus V(M)}\big(x-d_G(v)\big)\right).\end{equation}
Mohammadian  proved  that all roots  of $\mathscr{L\hspace{-0.7mm}M}(G,x)$ are real and nonnegative, and moreover,  if $\mathnormal{\Delta}(G)\geq2$,  then the zeros of
$\mathscr{L\hspace{-0.7mm}M}(G,x)$ lie in the interval $[0, \mathnormal{\Delta}(G)+2\sqrt{\mathnormal{\Delta}(G)-1})$.
By observing this interval, it  is   natural to ask: What is the sufficient and necessary condition for $0$  is a root of $\mathscr{L\hspace{-0.7mm}M}(G,x)$?
More generally, as a new real rooted graph   polynomial, it is natural to investigate the properties of  zeros
such as the interlacing of zeros, the upper and lower bounds of the  largest zero, the maximum multiplicity of   zeros, and the number of distinct zeros.
In this paper, we mainly  prove that the assertions given in  \eqref{Lmatchingthm1}--\eqref{Lmatchingthm3}  hold for   any  connected graph $G$, letting  $\ell(G)$ be   the length of the longest path in $G$.
\begin{idea}\label{Lmatchingthm1}
If $\mathnormal{\Delta}(G)\geq 2$, then the zeros of $\mathscr{L\hspace{-0.7mm}M}(G,x)$ are contained  in
the interval $[0, \mathnormal{\Delta}(G)+2\sqrt{\mathnormal{\Delta}(G)-1}\cos\tfrac{\pi}{2\ell(G)+2}]$,   and  in addition,  the upper bound of the   interval is a zero of $\mathscr{L\hspace{-0.7mm}M}(G, x)$ if and only if $G$ is a   cycle.
\end{idea}
\begin{idea}\label{Lmatchingthm4}
The number of distinct  positive roots   of $\mathscr{L\hspace{-0.7mm}M}(G,x)$ is at least equal to $\ell(G)$. Also, if $\delta(G)\geq2$, then  $\mathscr{L\hspace{-0.7mm}M}(G, x)$ has   at least  $\ell(G)+1$  distinct  positive roots.
\end{idea}
\begin{idea}\label{Lmatchingthm3}
For each edge $e\in E(G)$,  the zeros of  $\mathscr{L\hspace{-0.7mm}M}(G,x)$ and $\mathscr{L\hspace{-0.7mm}M}(G-e,x)$  interlace  in the sense that,  if    $\alpha_1\leq\cdots\leq\alpha_n$ and
$\beta_1\leq\cdots\leq\beta_n$ are respectively  the zeros  of $\mathscr{L\hspace{-0.7mm}M}(G,x)$ and $\mathscr{L\hspace{-0.7mm}M}(G-e,x)$ in which     $n=|V(G)|$,  then $\beta_1\leq\alpha_1\leq\beta_2\leq\alpha_2\leq\cdots\leq\beta_n\leq\alpha_n$. Further, the largest zero of $\mathscr{L\hspace{-0.7mm}M}(G,x)$  has the  multiplicity $1$  and is strictly  greater than the largest zero of $\mathscr{L\hspace{-0.7mm}M}(H,x)$ for any proper subgraph $H$ of $G$.
\end{idea}

It should be mentioned that the Laplacian matching polynomial is  recently  studied  under a  different name and
expression   by   Chen and Zhang  \cite{Zhang}.

For a graph $G$, the  {\it   subdivision} of $G$,  denoted by $S(G)$, is   the graph  derived from $G$ by replacing every  edge $e=\{a, b\}$ of    $G$  with  two edges $\{a, \upsilon_e\}$ and $\{\upsilon_e, b\}$  along with the new vertex   $\upsilon_e$ corresponding to the edge $e$. We know from a result of Yan and  Yeh  \cite{Yan}  that
\begin{equation}\label{subx}\mathscr{M}\big(S(G), x\big)=x^{|E(G)|-|V(G)|}\mathscr{L\hspace{-0.7mm}M}(G, x^2)\end{equation}
for any graph $G$, which is also proved by  Chen and Zhang  \cite{Zhang} by different method.
The equality  \eqref{subx} shows that the problem of the   location of zeros   of the Laplacian matching polynomial of a   graph $G$ can be  transformed into the problem that deals with the   location of zeros   of the matching polynomial of $S(G)$.
For an instance,  using   \eqref{subx}  and  the first   statement in   \eqref{thmmatching1}, it  immediately follows   that the zeros of $\mathscr{L\hspace{-0.7mm}M}(G, x)$ are nonnegative real numbers.
The assertion  \eqref{Lmatchingthm4}
is    proved in  Section  \ref{sec2} by    the    subdivision of graphs.

One of the most important   tools  in the theory of the   matching polynomial   is  the concept of   `path-tree'  which is  introduced  by Godsil \cite{Godsil2}.
Given a    graph $G$ and a vertex $u\in V(G)$,
the {\it   path-tree   $T(G, u)$} is the tree which has as vertices the paths in $G$  which start
at $u$  where two such paths
are adjacent if one is a maximal proper subpath of the other.
In Section  \ref{sec3}, we show that the path-tree   is   also applicable  for the Laplacian matching polynomial by making some  appropriate adjustments.
Using this,  we   prove    \eqref{Lmatchingthm1}
which is a slight improvement of the second statement  of Theorem 2.6 of    \cite{Ali}.
The assertion  \eqref{Lmatchingthm3} is      proved in  Section  \ref{sec3} by     linear algebra arguments.

Let us introduce more notations and definitions  before moving on to the next section.
We  use $\lambda(f(x))$ to denote the largest zero   of a real rooted polynomial $f(x)$.
For a square matrix  $M$, we shall use $\varphi(M, x)$ to denote the characteristic polynomial of $M$ in the indeterminate  $x$.
If all the roots  of $\varphi(M, x)$ are real, then its largest zero is denoted by $\lambda(M)$.
For a graph $G$, the  {\it  adjacency matrix} of  $G$, denoted by $A(G)$, is a matrix whose rows and columns are  indexed by $V(G)$ and the $(u, v)$-entry is $1$ if  $u$ and $v$ are adjacent  and $0$ otherwise.  Let $D(G)$ be  the diagonal matrix whose rows and columns are  indexed as the rows and the  columns of $A(G)$ with  $d_G(v)$ in  the $v$th diagonal position. The matrices  $L(G)=D(G)-A(G)$ and  $Q(G)=D(G)+A(G)$  are respectively   said to be  the {\it  Laplacian matrix} and the {\it  signless Laplacian matrix}  of $G$.
It is   known    that $\mathscr{M}(G, x)=\varphi(A(G), x)$    if and only if $G$ is a forest \cite{sac}.
In addition, it is proved  that  $\mathscr{L\hspace{-0.7mm}M}(G,x)=\varphi(L(G), x)$ if and only if $G$ is a forest   \cite{Ali}.
Among other results, we present a generalization of these results   in Section \ref{sec2}.

\section{Subdivision of graphs and the Laplacian matching polynomial}\label{sec2}

In this   section,  we   examine  the  location of zeros of  the Laplacian matching polynomial by  establishing
a relation between the Laplacian matching polynomial of a graph and the matching polynomial of the  subdivision of that  graph.
Then, by analysing the structures of the subdivision of graphs, we will prove   \eqref{Lmatchingthm4}.
To begin with, we recall  the multivariate matching polynomial that covers both  the matching polynomial and the Laplacian matching polynomial.
This multivariate
graph polynomial was introduced by Heilmann and Lieb \cite{Heilmann}.

Let $G$ be a graph and associate  the  vector    $\bmi{x}_G=(x_v)_{v\in V(G)}$ with  $G$ in which  $x_v$ is an  indeterminate    corresponding to   the vertex    $v\in V(G)$.
Notice  that, for a    subgraph $H$ of $G$,    $\bmi{x}_H$ is the vector that has the same coordinate as   $\bmi{x}_G$ in the positions corresponding  to the vertices   in $V(H)$.
The {\it   multivariate matching polynomial} of $G$  is defined as
\begin{equation}\label{defx}\mathfrak{M}(G, \bmi{x}_G)=\sum_{M\in\mathcal{M}(G)}(-1)^{|M|}\left(\prod_{v\in V(G)\setminus V(M)}x_v\right).\end{equation}
Let  $\bmi{1}_G$ be the all one  vector of length    $|V(G)|$.
Also, for  a   subgraph  $H$ of $G$, we let
$\bmi{d}_{G, H}=(d_G(v))_{v\in V(H)}$.
For simplicity, we write $\bmi{d}_G$ instead of $\bmi{d}_{G, G}$.
We sometimes  drop the subscript  of the vector symbols if there is no possible confusion.
It is easy to see that
\begin{equation}\label{1x}\mathfrak{M}(G, x\bmi{1}_G)=\mathscr{M}(G,x)\end{equation}
and
\begin{equation}\label{lapx}\mathfrak{M}(G, x\bmi{1}_G-\bmi{d}_G\big)=\mathscr{L\hspace{-0.7mm}M}(G,x).\end{equation}
Note that
$$\mathfrak{M}\big(G_1\cup G_2, (\bmi{x}_{G_1}, \bmi{x}_{G_2})\big)=\mathfrak{M}(G_1, \bmi{x}_{G_1})\mathfrak{M}(G_2, \bmi{x}_{G_2}),$$
where $G_1\cup G_2$ denotes the disjoint union of two graphs  $G_1$ and $G_2$.
So, in what follows, we  often   restrict our attention on   connected graphs.

We need  the following useful lemma in the sequel.

\begin{lemma}[Amini \cite{Amini}]\label{basiclemma}
Let $G$ be a graph. For any vertex  $v\in V(G)$,
$$\mathfrak{M}(G, \bmi{x}_G)=x_v\mathfrak{M}(G-v,\bmi{x}_{G-v})-\sum_{w\in N_G(v)}\mathfrak{M}(G-v-w,\bmi{x}_{G-v-w}).$$
\end{lemma}

By combining Lemma \ref{basiclemma} and \eqref{1x}, we get
\begin{equation}\label{var1x}\mathscr{M}(G, x)=x\mathscr{M}(G-v,x)-\sum_{w\in N_G(v)}\mathscr{M}(G-v-w,x),\end{equation}
which is a well known recursive formula for the matching polynomial.

The following theorem, which is a generalization of   \eqref{subx},   plays a crucial role in our proofs in Section \ref{sec3}.

\begin{theorem}\label{subdivisionlemma}
Let $G$ be a graph.
For any   subset $W$ of $V(G)$,
$$\mathscr{M}\big(S(G)-W, x\big)=x^{|E(G)|-|V(G)|+|W|}\mathfrak{M}(G-W, x^2\bmi{1}_{G-W}-\bmi{d}_{G, G-W}).$$
\end{theorem}

\begin{proof}
For simplicity, let  $k=|V(G)\setminus W|$ and $m=|E(G)|$.
We prove the assertion  by induction on $k$.
If   $V(G)\setminus W=\{u\}$ for some vertex $u\in V(G)$, then $S(G)-W$ consists of a star  on    $d_G(u)+1$ vertices  and $|E(G)|-d_G(u)$ isolated vertices.
Therefore,
$$\mathscr{M}\big(S(G)-W,x\big)=x^{m+1}-d_G(u)x^{m-1}$$
and
$$\mathfrak{M}\left(G-W, x^2\bmi{1}-\bmi{d}\right)=x^2-d_G(u).$$
So, the claimed equality holds for  $k=1$.
Assume that $k\geq2$.
Choose a vertex $u\in V(G)\setminus W$ and let $H=S(G)-W-u$.
By  Lemma  \ref{basiclemma}, the induction   hypothesis and \eqref{var1x}, we have
\begin{align*}
x^{m-k+2}\mathfrak{M}(G-W, x^2\bmi{1}-\bmi{d})&=x\big(x^2-d_G(u)\big)x^{m-k+1}\mathfrak{M}(G-W-u, x^2\bmi{1}-\bmi{d})\\
&-\sum_{v\in N_{G-W}(u)}x^{m-k+2}\mathfrak{M}(G-W-u-v, x^2\bmi{1}-\bmi{d})\\
&=x\big(x^2-d_G(u)\big)\mathscr{M}(H,x)
-\sum_{v\in N_{G-W}(u)}\mathscr{M}(H-v,x)\\
&=x^2\mathscr{M}\big(S(G)-W,x\big)
+x^2\sum_{v\in N_{S(G)-W}(u)}\mathscr{M}(H-v,x)\\
&-d_G(u)x\mathscr{M}(H,x)
-\sum_{v\in N_{G-W}(u)}\mathscr{M}(H-v,x).
\end{align*}
Hence, in order to complete the induction step,   it suffices to prove that
\begin{equation}\label{eqx}
d_G(u)x\mathscr{M}(H,x)=x^2\sum_{v\in N_{S(G)-W}(u)}\mathscr{M}(H-v,x)-\sum_{v\in N_{G-W}(u)}\mathscr{M}(H-v,x).
\end{equation}
To establish  \eqref{eqx},
let $N_G(u)\cap W=\{a_1, \ldots, a_s\}$
and $N_G(u)\setminus W=\{b_1, \ldots, b_t\}$.
Also, for  $i=1, \ldots, s$, let $a'_i$ be   the vertex of $S(G)$ corresponding to the edge $\{u, a_i\}$ of $G$  and,  for $j=1, \ldots, t$, let $b'_j$  be   the vertex of $S(G)$ corresponding to the edge $\{u, b_j\}$ of $G$.
Notice  that, if one of
$N_G(u)\cap W$
and $N_G(u)\setminus W$ is empty,
then we may derive \eqref{eqx} by the  same discussion as below.
We have  $d_G(u)=s+t$ and $N_{S(G)-W}(u)=N_{S(G)}(u)=\{a'_1, \ldots, a'_s, b'_1, \ldots, b'_t\}$.
The structure of $H$ is illustrated  in  Figure \ref{figM1}.

\begin{figure}[htb]
\centering
\begin{tikzpicture}
\draw[thick]   (8,-1.2) ellipse[x radius=3.8,y radius=1];
\draw[thin] (9.7,-1.8)--(10.5,-2.3)node[right] {$|E(G[W])|$ isolated vertices};
\draw[thick, dotted]   (8,-1.2) ellipse[x radius=3.3,y radius=0.5];
\draw[thin] (7,-1.5)--(5,-2.3)node[left] {Vertices of $W$};
\filldraw[fill=black,draw=black] (5.4,   -1.3)  circle [radius=0.07] node[above] {$a_1$};
\filldraw[fill=black,draw=black] (6.1,   -1.3)  circle [radius=0.07] node[above] {$a_2$};
\filldraw[fill=black,draw=black] (6.8,   -1.3)  circle [radius=0.07] node[above] {$a_3$};
\filldraw[fill=black,draw=black] (9,     -1.3)  circle [radius=0.07] node[above] {$a_{s-1}$};
\filldraw[fill=black,draw=black] (9.6,   -1.3)  circle [radius=0.07] node[above] {$a_s$};
\filldraw[fill=black,draw=black] (7.6,   -1.3)  circle [radius=0.02];
\filldraw[fill=black,draw=black] (8,   -1.3)  circle [radius=0.02];
\filldraw[fill=black,draw=black] (8.4,   -1.3)  circle [radius=0.02];
\draw[thick,dotted] (6,   -3)--(5.4,   -1.3);
\draw[thick,dotted] (6.6,   -3)--(6.1,   -1.3);
\draw[thick,dotted] (7.2,   -3)--(6.8,   -1.3);
\draw[thick,dotted] (8.8,   -3)--(9,   -1.3);
\draw[thick,dotted] (9.4,   -3)--(9.6,   -1.3);
\filldraw[fill=black,draw=black] (6,    -3)  circle [radius=0.07] node[above] {$a'_1$};
\filldraw[fill=black,draw=black] (6.6,  -3)  circle [radius=0.07] node[above] {$a'_2$};
\filldraw[fill=black,draw=black] (7.2,  -3)  circle [radius=0.07] node[above] {$a'_3$};
\filldraw[fill=black,draw=black] (8.8,  -3)  circle [radius=0.07] node[above] {$a'_{s-1}$};
\filldraw[fill=black,draw=black] (9.4,  -3)  circle [radius=0.07] node[above] {$a'_s$};
\filldraw[fill=black,draw=black] (7.6,   -3)  circle [radius=0.02];
\filldraw[fill=black,draw=black] (8,     -3)  circle [radius=0.02];
\filldraw[fill=black,draw=black] (8.4,   -3)  circle [radius=0.02];
\draw[thick, dotted] (6,   -3)--(8,     -4);
\draw[thick, dotted] (6.6,   -3)--(8,     -4);
\draw[thick, dotted] (7.2,   -3)--(8,     -4);
\draw[thick, dotted] (8.8,   -3)--(8,     -4);
\draw[thick, dotted] (9.4,   -3)--(8,     -4);
\filldraw[fill=black,draw=black] (8,     -4)  circle [radius=0.08] node[above] {$u$};
\draw[thick,dotted] (8,  -4)  circle [radius=0.4];
\filldraw[fill=black,draw=black] (6,   -5)  circle [radius=0.07] node[above] {$b'_1$};
\filldraw[fill=black,draw=black] (6.6, -5)  circle [radius=0.07] node[above] {$b'_2$};
\filldraw[fill=black,draw=black] (7.2, -5)  circle [radius=0.07] node[above] {$b'_3$};
\filldraw[fill=black,draw=black] (8.8, -5)  circle [radius=0.07] node[above] {$b'_{t-1}$};
\filldraw[fill=black,draw=black] (9.4, -5)  circle [radius=0.07] node[above] {$b'_t$};
\filldraw[fill=black,draw=black] (7.6, -5)  circle [radius=0.02];
\filldraw[fill=black,draw=black] (8,   -5)  circle [radius=0.02];
\filldraw[fill=black,draw=black] (8.4, -5)  circle [radius=0.02];
\draw[thick, dotted] (6,   -5)--(8,     -4);
\draw[thick, dotted] (6.6,   -5)--(8,     -4);
\draw[thick, dotted] (7.2,   -5)--(8,     -4);
\draw[thick, dotted] (8.8,   -5)--(8,     -4);
\draw[thick, dotted] (9.4,   -5)--(8,     -4);
\draw[very thick] (6,   -5)--(5.4,   -6.8);
\draw[very thick] (6.6,   -5)--(6.1,   -6.8);
\draw[very thick] (7.2,   -5)--(6.8,   -6.8);
\draw[very thick] (8.8,   -5)--(9,   -6.8);
\draw[very thick] (9.4,   -5)--(9.6,   -6.8);
\draw[thick]  (8,-6.8) ellipse[x radius=3.8,y radius=1];
\draw[ thin] (10,-6.5)--(10.5,-5.8)node[right] {$S(G-W-u)$};
\filldraw[fill=black,draw=black] (5.4,   -6.8)  circle [radius=0.07] node[below] {$b_1$};
\filldraw[fill=black,draw=black] (6.1,   -6.8)  circle [radius=0.07] node[below] {$b_2$};
\filldraw[fill=black,draw=black] (6.8,   -6.8)  circle [radius=0.07] node[below] {$b_3$};
\filldraw[fill=black,draw=black] (9,     -6.8)  circle [radius=0.07] node[below] {$b_{t-1}$};
\filldraw[fill=black,draw=black] (9.6,   -6.8)  circle [radius=0.07] node[below] {$b_t$};
\filldraw[fill=black,draw=black] (7.6,   -6.8)  circle [radius=0.02];
\filldraw[fill=black,draw=black] (8,     -6.8)  circle [radius=0.02];
\filldraw[fill=black,draw=black] (8.4,   -6.8)  circle [radius=0.02];
\end{tikzpicture}
\caption{The structure of  $H$.}\label{figM1}
\end{figure}
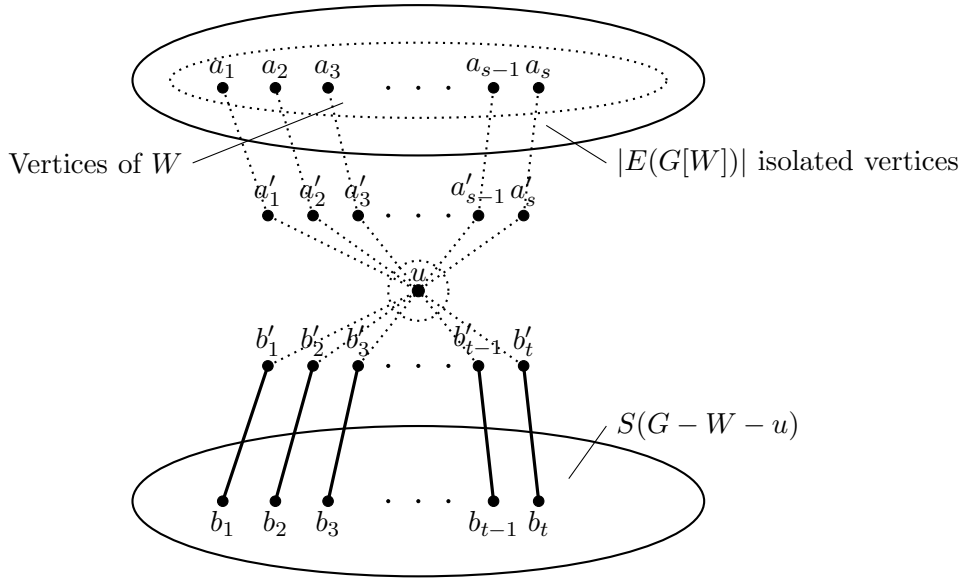

We have   $d_H(a'_i)=0$  for $i=1, \ldots, s$  and $d_H(b'_j)=1$ for $j=1, \ldots, t$. By applying \eqref{var1x} for    $a'_i$ and $b'_j$, we find that
$$\mathscr{M}(H,x)=x\mathscr{M}(H-a'_i,x)$$
and
\begin{align*}x\mathscr{M}(H, x)&=x^2\mathscr{M}(H-b'_j,x)-x\mathscr{M}(H-b_j-b'_j,x)\\&=x^2\mathscr{M}(H-b'_j,x)-\mathscr{M}(H-b_j,x).\end{align*}
Therefore,
\begin{align*}
d_G(u)x\mathscr{M}(H,x)&=sx\mathscr{M}(H,x)+tx\mathscr{M}(H,x)\\
&=x^2\sum_{i=1}^s\mathscr{M}(H-a'_i,x)+x^2\sum_{j=1}^t\mathscr{M}(H-b'_j,x)-\sum_{j=1}^t\mathscr{M}(H-b_j,x)\\
&=x^2\sum_{v\in N_{S(G)-W}(u)}\mathscr{M}(H-v,x)-\sum_{v\in N_{G-W}(u)}\mathscr{M}(H-v,x),
\end{align*}
which is exactly  \eqref{eqx}.  This  completes the proof.
\end{proof}

In what follows, we  prove some results about  the  Laplacian matching polynomial    by analysing the structures of the        subdivision of graphs.
The following consequence immediately follows from  Theorem \ref{subdivisionlemma} and  the first   statement in   \eqref{thmmatching1}.
It worth to mention that  the following result    is     proved in \cite{Zhang} for a different  expression of the Laplacian matching polynomial.

\begin{corollary}\label{XrootrelationX}
Let $G$ be a graph. Then  $$\mathscr{M}\big(S(G), x\big)=x^{|E(G)|-|V(G)|}\mathscr{L\hspace{-0.7mm}M}(G, x^2).$$
In particular, the zeros of $\mathscr{L\hspace{-0.7mm}M}(G, x)$ are nonnegative real numbers.
\end{corollary}

For  a  graph  $G$,  it is proved   that  $\mathscr{L\hspace{-0.7mm}M}(G, x)=\varphi(L(G), x)$ if and only if $G$ is a forest   \cite{Ali}.
Since $0$ is an eigenvalue of $L(G)$,  we deduce that   $\mathscr{L\hspace{-0.7mm}M}(G, 0)=0$ if  $G$ is a forest. From \eqref{defali}, we get  the   combinatorial identity
$$\sum_{M\in\mathcal{M}(F)}(-1)^{|M|}\left(\prod_{v\in V(F)\setminus V(M)}d_F(v)\right)=0$$
for any   forest $F$.
The following  theorem, which  is proved   in \cite{Zhang},   gives        a necessary and sufficient condition for  $0$  to be a root of the Laplacian matching polynomial.
We present  here a different proof for it.

\begin{theorem}[Chen, Zhang \cite{Zhang}]
Let $G$ be a connected graph.
Then,  $0$  is a root of $\mathscr{L\hspace{-0.7mm}M}(G, x)$ if and only if $G$ is a tree.
\end{theorem}

\begin{proof}
If $G$ is a tree, then $|E(G)|=|V(G)|-1$ and so $\mathscr{L\hspace{-0.7mm}M}(G, x^2)=x\mathscr{M}(S(G), x)$  by   Corollary  \ref{XrootrelationX}, implying that  $0$  is a root of $\mathscr{L\hspace{-0.7mm}M}(G, x)$.
We   prove that  $0$  is not a root of $\mathscr{L\hspace{-0.7mm}M}(G, x)$  if $G$ is not a tree.
For this, assume that  $|E(G)|\geq|V(G)|$.
One may easily  consider  $S(G)$ as  a bipartite graph with the bipartition $\{V(G), E(G)\}$ after  identifying each  new vertex   $\upsilon_e$ of $S(G)$ with its  corresponding   edge $e$ of $G$.

We claim that $S(G)$ has a matching    that saturates   the part  $V(G)$.
If $G$ contains a vertex $u$ with degree $1$  and  $e$ is the edge incident to $u$ in $G$, then   it suffices to prove that $S(G-u)$ has a matching  that saturates   the part   $V(G-u)$,  since the union of such matching   and the edge $\{u, \upsilon_e\}$ forms  a matching of $S(G)$  that saturates   the part    $V(G)$.
Thus,  we may assume that $d_G(v)\geq 2$ for all vertices  $v\in V(G)$.
We are going to establish  that  $S(G)$ satisfies   Hall's condition    \cite[Theorem 16.4]{244}.
For a subset  $W$ of $V(G)$, we shall use $N_G(W)$ to denote the set of vertices of $G$ each of which is adjacent  to a vertex in $W$ and   $\partial_G(W)$  to denote the set of edges of $G$ each of which has  exactly  one endpoint  in $W$.
For any  subset $U$ of the part $V(G)$, since $d_G(v)\geq 2$ for all vertices $v\in V(G)$,
\begin{equation}\label{H1x}|\partial_{S(G)}(U)|\geq2|U|.\end{equation}
On the other hand, $d_{S(G)}(\upsilon_e)=2$ for each  $e\in E(G)$,   so
\begin{equation}\label{H2x}\big|\partial_{S(G)}\big(N_{S(G)}(U)\big)\big|=2|N_{S(G)}(U)|.\end{equation}
Clearly,   $|\partial_{S(G)}(N_{S(G)}(U))|\geq|\partial_{S(G)}(U)|$ which implies that    $|N_{S(G)}(U)|\geq|U|$ using  \eqref{H1x} and \eqref{H2x}.  This means that   $S(G)$ satisfies   Hall's condition, as required.

We proved that $S(G)$ has a   matching    that saturates   the part  $V(G)$. This means that the   smallest power of $x$ in $\mathscr{M}(S(G),x)$ is $|E(G)|-|V(G)|$  by \eqref{defx} and \eqref{1x}. In view of  Corollary  \ref{XrootrelationX}, $\mathscr{M}(S(G), x)=x^{|E(G)|-|V(G)|}\mathscr{L\hspace{-0.7mm}M}(G, x^2)$
which shows that  the constant term in  $\mathscr{L\hspace{-0.7mm}M}(G, x)$ is nonzero.
So,  $0$  is not a root of $\mathscr{L\hspace{-0.7mm}M}(G, x)$. This completes the proof.
\end{proof}

In the next theorem, we give  a   lower   bound  on  the number of distinct  zeros   of the Laplacian matching polynomial.

\begin{theorem}\label{l-w}
Let $G$ be a connected graph  and let $\ell(G)$ be the length of the longest path in $G$. Then the number of distinct  positive roots  of $\mathscr{L\hspace{-0.7mm}M}(G,x)$ is at least equal to $\ell(G)$.
Also, if $\delta(G)\geq2$, then  $\mathscr{L\hspace{-0.7mm}M}(G,x)$ has   at least  $\ell(G)+1$  distinct  positive roots.
\end{theorem}

\begin{proof}
For  convenience, let $\ell=\ell(G)$. Denote by $\ell'$  the length of the longest path in $S(G)$.
From   \eqref{thmmatching3},  $\mathscr{M}(S(G),x)$ has at least $\ell'+1$ distinct roots.
By   Corollary  \ref{XrootrelationX}, $\mathscr{M}(S(G), x)=x^{|E(G)|-|V(G)|}\mathscr{L\hspace{-0.7mm}M}(G, x^2)$
which shows that  $\mathscr{L\hspace{-0.7mm}M}(G, x^2)$ has at least $\ell'$ distinct nonzero roots. Since all roots of   $\mathscr{L\hspace{-0.7mm}M}(G, x)$  are real and nonnegative by    Corollary  \ref{XrootrelationX},   it follows that    $\mathscr{L\hspace{-0.7mm}M}(G, x)$  has at least $\lceil\ell'/2\rceil$ distinct positive roots.

For each  edge    $e\in E(G)$,   denote by    $\upsilon_e$     the vertex of $S(G)$  corresponding to   $e$.
Let    $w_0, w_1, \ldots, w_\ell$ be a path in  $G$.  Then,  $w_0, \upsilon_{e_1}, w_1,    \ldots,  \upsilon_{e_\ell}, w_\ell$ is a path in $S(G)$ of   length $2\ell$, where $e_i=\{w_{i-1}, w_i\}\in E(G)$ for $i=1, \ldots, \ell$. Thus,   $\ell'\geq2\ell$  and so $\mathscr{L\hspace{-0.7mm}M}(G, x)$   has at least $\ell$ distinct positive roots.

Now, assume that $\delta(G)\geq2$.  This assumption allows us to consider  a vertex   $w'\in N_G(w_0)\setminus\{w_1\}$.  Then,   $S(G)$ contains  the path $\upsilon_{e'}, w_0,  \upsilon_{e_1}, w_1,    \ldots,  \upsilon_{e_\ell}, w_\ell$  of   length $2\ell+1$, where  $e'=\{w', w_0\}\in E(G)$. Therefore,   $\ell'\geq2\ell+1$  and so  $\mathscr{L\hspace{-0.7mm}M}(G, x)$   has at least $\lceil\ell'/2\rceil\geq\ell+1$ distinct positive roots. This completes the proof.
\end{proof}

\begin{remark}
The second statement in Theorem \ref{l-w} implies  that,  if $G$
is a graph with a
Hamilton cycle, then the zeros of  $\mathscr{L\hspace{-0.7mm}M}(G,x)$ are all distinct.
\end{remark}

Given a  graph  $G$,
it is   known    that $\mathscr{M}(G, x)=\varphi(A(G, x)$    if and only if $G$ is a forest \cite{sac}.
Also, as we mentioned before, it is  established that  $\mathscr{L\hspace{-0.7mm}M}(G, x)=\varphi(L(G), x)$ if and only if $G$ is a forest   \cite{Ali}.
Below, we present a general  result   which shows that the multivariate matching polynomial of a forest has a  determinantal representation  in terms  of    its adjacency matrix,  which will be used in the next section.

\begin{theorem}\label{treeequal}
Let  $F$ be  a forest. Then $\mathfrak{M}(F, \bmi{x}_F)=\det(\bmi{X}_F-A(F))$,
where $\bmi{X}_F$ is a diagonal matrix
whose rows and columns are indexed by $V(F)$ and the $(v,  v)$-entry is
$x_v$ for any vertex   $v\in V(F)$.
In particular, $\mathscr{M}(F, x)=\varphi(A(F), x)$ and $\mathscr{L\hspace{-0.7mm}M}(F,x)=\varphi(L(F), x)$.
\end{theorem}

\begin{proof}
We prove that  $\mathfrak{M}(F, \bmi{x}_F)=\det(\bmi{X}_F-A(F))$  by induction on   $|E(F)|$.
The equality  is trivially valid if   $|E(F)|=0$.
So, assume that   $|E(F)|\geq1$. As $F$ is a forest, we may   consider two vertices  $u, v\in V(F)$ with $N_F(u)=\{v\}$.
Without loss of generality, we may assume that the first row and column of $A(F)$ are
corresponding to  $u$ and  the second  row and column of $A(F)$ are
corresponding to  $v$.
Expanding the determinant of $\bmi{X}_F-A(F)$ along its  first row, we obtain  by the induction hypothesis and Lemma \ref{basiclemma} that
\begin{align*}
\det\big(\bmi{X}_F-A(F)\big)&=x_u\det\big(\bmi{X}_{F-u}-A(F-u)\big)-\det\big(\bmi{X}_{F-u-v}-A(F-u-v)\big)\\&=x_u\mathfrak{M}(F-u, \bmi{x}_{F-u})-\mathfrak{M}(F-u-v, \bmi{x}_{F-u-v})\\&=\mathfrak{M}(F, \bmi{x}_F),
\end{align*}
as desired. The `in particular' statement  immediately  follows     from \eqref{1x} and \eqref{lapx}.
\end{proof}

\begin{corollary}
For a tree  $T$,  the  multiplicity of $0$ as a root  of $\mathscr{L\hspace{-0.7mm}M}(T, x)$ is $1$.
\end{corollary}

\begin{proof}
It is well known that the number of connected components of a  graph  $\mathnormal{\Gamma}$ is equal to the  multiplicity of  $0$ as a root of $\varphi(L(\mathnormal{\Gamma}), x)$  \cite[Proposition 1.3.7]{Brouwer}.
So, the result follows from    $\mathscr{L\hspace{-0.7mm}M}(T,x)=\varphi(L(T), x)$  which is given in Theorem   \ref{treeequal}.
\end{proof}

\section{The largest zero of the Laplacian matching polynomial}\label{sec3}

The purpose of this section is to investigate  the location of the largest zero of the Laplacian matching polynomial.
We give
a linear algebraic approach to study
the largest zero of the Laplacian matching polynomial  and  present   sharp  upper  and  lower bounds on it.
The assertions   \eqref{Lmatchingthm1} and \eqref{Lmatchingthm3}     are also proved in this section based on the  linear algebraic approach.

Let $G$ be a connected graph and   $u\in V(G)$. Let $T(G,u)$ be the path-tree of $G$ respect to  the vertex  $u$ which is introduced  in Section \ref{sec1}.
Consider two vectors        $\bmi{x}_G=(x_v)_{v\in V(G)}$ and  $\bmi{x}_{T(G,u)}=(x_P)_{P\in V(T(G,u))}$   of indeterminates associated with  $G$ and $T(G,u)$, respectively.
For  every  vertex       $P\in V(T(G, u))$, we  may  identify  $x_P$ with  $x_{\upsilon(P)}$ in which  $\upsilon(P)$  is the terminal vertex of the path $P$ in $G$.
In such way,    $G$ and  $T(G,u)$ will be equipped with two vectors consisting of the same indeterminates, which are simply denoted by $\bmi{x}$ when  there is no ambiguity.
In  what follows,   for every  subgraph $H$ of $G$ and vertex $u\in V(H)$,
we   denote by  $D_G(T(H, u))$
the  diagonal matrix
whose rows and columns are indexed by $V(T(H, u))$ and the $(P,  P)$-entry is
$d_G(\upsilon(P))$.

The univariate version of the following theorem, which is proved by Godsil  \cite{Godsil2}, has a key role in  the theory of the matching polynomial.
Notice  that, for a graph $G$ and a  vertex $u\in V(G)$,     $u$ is   a path   in $G$  and the corresponding vertex in  $T(G, u)$  will also be referred to as $u$.

\begin{theorem}[Amini \cite{Amini}]\label{d1}
Let $G$ be a connected graph and  let   $u\in V(G)$.
Then
$$\frac{\mathfrak{M}(G-u,\bmi{x})}{\mathfrak{M}(G, \bmi{x})}=\frac{\mathfrak{M}(T(G,u)-u, \bmi{x})}{\mathfrak{M}(T(G,u), \bmi{x})},$$
and moreover, $\mathfrak{M}(G,\bmi{x})$ divides $\mathfrak{M}(T(G,u), \bmi{x})$.
\end{theorem}

For a connected graph  $G$ and a  vertex   $u\in V(G)$,   Theorem \ref{d1} and    Theorem  \ref{treeequal}  yield   that $\mathscr{M}(G, x)$ divides $\varphi(A(T(G,u)), x)$.
Since   all roots of  the characteristic polynomial of a   symmetric matrix are  real, the first statement in  \eqref{thmmatching1} is obtained   as an application of Theorem \ref{d1}.
For  the Laplacian  matching polynomial,
we get  the following   result.

\begin{corollary}\label{divide}
Let $G$ be a  connected  graph, $H$  be a   subgraph  of $G$,  and   $u\in V(H)$.
If $H$ is connected, then
$\mathfrak{M}(H, x\bmi{1}_H-\bmi{d}_{G, H})$ divides  $\varphi(D_G(T(H, u))+A(T(H, u)), x)$.
In particular, $\varphi(D_G(T(G, u))+A(T(G,u)), x)$ is divisible by $\mathscr{L\hspace{-0.7mm}M}(G, x)$
for  every vertex $u\in V(G)$.
\end{corollary}

\begin{proof}
By Theorem \ref{d1}, we find that
$\mathfrak{M}(H, x\bmi{1}_H-\bmi{d}_{G, H})$ divides $\mathfrak{M}(T(H, u), x\bmi{1}_H-\bmi{d}_{G, H})$.
It follows from Theorem  \ref{treeequal} that
\begin{align*}\mathfrak{M}\big(T(H, u), x\bmi{1}_H-\bmi{d}_{G, H}\big)&=\det\Big(xI-D_G\big(T(H, u)\big)-A\big(T(H, u)\big)\Big)\\&=\varphi\Big(D_G\big(T(H, u)\big)+A\big(T(H, u)\big), x\Big),\end{align*}
which establishes  what we require.
Since     $\mathfrak{M}(G, x\bmi{1}_G-\bmi{d}_G)=\mathscr{L\hspace{-0.7mm}M}(G, x)$   using    \eqref{lapx}, the `in particular' statement  immediately  follows.
\end{proof}

\begin{remark}
The matrix $D_G(T(G, u))+A(T(G,u))$, which   appeared     in  Corollary \ref{divide},  is a  symmetric diagonally dominant  matrix with nonnegative diagonal entries, so   all of its eigenvalues are  nonnegative real numbers.
Hence, Corollary \ref{divide} gives us  another proof  for the fact that  all roots  of the Laplacian  matching polynomial  are real and  nonnegative which  was    also proved in Corollary \ref{XrootrelationX}.
\end{remark}

It is    well known   that the largest zero of the matching polynomial of a graph is equal to the largest eigenvalue of the adjacency matrix of a path-tree of that graph. This fact is obtained by combining the Perron--Frobenius theorem  \cite[Theorem 2.2.1]{Brouwer} and   Theorems    \ref{treeequal}   and \ref{d1}.
The following theorem  can be considered  as  an analogue of the fact.
Indeed, the following theorem   presents    a linear algebra technique  to treat  with  the largest zero of the Laplacian matching polynomial.

\begin{theorem}\label{largestzero}
Let $G$ be a  connected  graph, $H$  be a   subgraph  of $G$,  and   $u\in V(H)$.
If $H$ is connected, then
\begin{equation}\label{largex}\lambda\big(\mathfrak{M}(H, x\bmi{1}_H-\bmi{d}_{G, H})\big)=\lambda\Big(D_G\big(T(H, u)\big)+A\big(T(H, u)\big)\Big).\end{equation}
In particular,
$\lambda(\mathscr{L\hspace{-0.7mm}M}(G, x))=\lambda(D_G(T(G, u))+A(T(G, u)))$. Also, the largest root  of $\mathscr{L\hspace{-0.7mm}M}(G, x)$ has the  multiplicity   $1$.
\end{theorem}

\begin{proof}
We prove  \eqref{largex}  by induction on  $|V(H)|$.
Clearly,  \eqref{largex}       is valid  for   $|V(H)|=1$.
Assume that  $|V(H)|\geq2$.
We  first  show  that
\begin{equation}\label{claimx}\lambda\big(\mathfrak{M}(H-u, x\bmi{1}_{H-u}-\bmi{d}_{G, H-u})\big)<\lambda\big(\mathfrak{M}(H, x\bmi{1}_H-\bmi{d}_{G, H})\big).\end{equation}
To see  \eqref{claimx},  we apply     Theorem \ref{subdivisionlemma} and \eqref{thmmatching2} to get that
\begin{align*}
\lambda\big(\mathfrak{M}(H-u, x^2\bmi{1}_{H-u}-\bmi{d}_{G, H-u})\big)&=\lambda\Big(\mathscr{M}\big(S(G)-W-u, x\big)\Big)\\
&<\lambda\Big(\mathscr{M}\big(S(G)-W, x\big)\Big)\\
&=\lambda\big(\mathfrak{M}(H, x^2\bmi{1}_H-\bmi{d}_{G, H})\big),
\end{align*}
where $W=V(G)\setminus V(H)$.
This clearly  proves    \eqref{claimx}.
Now, let $N_H(u)=\{u_1, \ldots, u_k\}$ and let  $H_i$ be the connected component  of $H-u$ containing $u_i$
for $i=1, \ldots, k$. By the  induction hypothesis,
\begin{equation}\label{largexfori}\lambda\big(\mathfrak{M}(H_i, x\bmi{1}_{H_i}-\bmi{d}_{G, H_i})\big)=\lambda\Big(D_G\big(T(H_i, u_i)\big)+A\big(T(H_i, u_i)\big)\Big)\end{equation}
for   $i=1, \ldots, k$.
It is not hard to see  the $k\times k$  block diagonal matrix whose $i$th block diagonal entry is $D_G(T(H_i, u_i))+A(T(H_i, u_i))$,  say  $R$,
is a   principal     submatrix of
$D_G(T(H, u))+A(T(H, u))$ with size $|T(H, u)|-1$.
Hence,  by the interlacing theorem \cite[Corollary 2.5.2]{Brouwer}, it follows that
$\lambda(R)$ is  greater than or equal to the second largest eigenvalue   of $D_G(T(H, u))+A(T(H, u))$.
Further, it follows from   \eqref{largexfori}  and  \eqref{claimx}  that
\begin{align*}\lambda(R)&=\max\Big\{\lambda\big(\mathfrak{M}(H_i, x\bmi{1}_{H_i}-\bmi{d}_{G, H_i})\big) \, \Big|  \,  1\leq i\leq k\Big\}\\
&=\lambda\big(\mathfrak{M}(H-u, x\bmi{1}_{H-u}-\bmi{d}_{G, H-u})\big)\\
&<\lambda\big(\mathfrak{M}(H, x\bmi{1}_H-\bmi{d}_{G, H})\big).\end{align*}
Thus,    $\lambda(\mathfrak{M}(H, x\bmi{1}_H-\bmi{d}_{G, H}))$ is strictly  greater than the second largest eigenvalue   of $D_G(T(H, u))+A(T(H, u))$. On  the other hand,
Corollary \ref{divide}  implies  that $\lambda(\mathfrak{M}(H, x\bmi{1}_H-\bmi{d}_{G, H}))$ is  a zero    of $\varphi(D_G(T(H, u))+A(T(H,u)), x)$. So, we conclude that     $\lambda(\mathfrak{M}(H, x\bmi{1}_H-\bmi{d}_{G, H}))$ is the largest eigenvalue   of $D_G(T(H, u))+A(T(H, u))$. This  completes   the induction step and  demonstrates that      \eqref{largex} holds.

For  the  `in particular' statement,  note that    \eqref{largex} and   \eqref{lapx} yield that
$$\lambda\Big(D_G\big(T(G, u)\big)+A\big(T(G, u)\big)\Big)=\lambda\big(\mathfrak{M}(G, x\bmi{1}_G-\bmi{d}_G)\big)=\lambda\big(\mathscr{L\hspace{-0.7mm}M}(G, x)\big),$$ and further,
the connectedness of  $G$   implies   that     $D_G(T(G, u))+A(T(G, u))$ is an    irreducible matrix with nonnegative entries,   and  consequently,
its largest eigenvalue  has the multiplicity $1$  by the Perron--Frobenius theorem  \cite[Theorem 2.2.1]{Brouwer}.
\end{proof}

\begin{corollary}\label{geqsignless}
Let $G$ be a connected graph and $u\in V(G)$. Then
\begin{equation}\label{eX}\lambda\big(\mathscr{L\hspace{-0.7mm}M}(G, x)\big)\geq\lambda\Big(L\big(T(G, u)\big)\Big)\end{equation}
with the   equality holds if and only  if   $G$ is a tree.
\end{corollary}

\begin{proof}
We first recall  the  fact that a graph $\mathnormal{\Gamma}$ is bipartite if and only if $\varphi(L(\mathnormal{\Gamma}), x)=\varphi(Q(\mathnormal{\Gamma}), x)$   \cite[Proposition 1.3.10]{Brouwer}.
For each  $P\in V(T(G, u))$, we   have $d_{T(G, u)}(P)\leq d_G(\upsilon(P))$, where  $\upsilon(P)$ is the terminal vertex of the path $P$ in $G$.
Therefore,  $R=D_G(T(G, u))+A(T(G, u))-Q(T(G, u))$ has  nonnegative entries, and thus,   Theorem \ref{largestzero},   the Perron--Frobenius theorem  \cite[Theorem 2.2.1]{Brouwer}, and the above mentioned fact   yield that
\begin{align}\label{errorx}\nonumber\lambda\big(\mathscr{L\hspace{-0.7mm}M}(G, x)\big)&=\lambda\Big(D_G\big(T(G, u)\big)+A\big(T(G, u)\big)\Big)\\\nonumber
&=\lambda\Big(R+Q\big(T(G, u)\big)\Big)\\
&\geq\lambda\Big(Q\big(T(G, u)\big)\Big)\\\nonumber
&=\lambda\Big(L\big(T(G, u)\big)\Big),\end{align}
proving \eqref{eX}.
If $G$ is a tree,  then $G$ is isometric to $T(G, u)$ and since $\mathscr{L\hspace{-0.7mm}M}(G, x)=\varphi(L(G), x)$ by Theorem \ref{treeequal}, the equality  in  \eqref{eX} is attained.
Conversely, assume that the equality    in  \eqref{eX} holds. Consequently,   the equality    in  \eqref{errorx} occurs, and hence,   the Perron--Frobenius theorem  \cite[Theorem 2.2.1]{Brouwer}  implies that    $R=0$.
This means that  $d_{T(G, u)}(P)=d_G(\upsilon(P))$ for each $P\in V(T(G, u))$.
We assert that $G$ is a tree.
Towards a  contradiction, suppose that  there is a   cycle $C$ in  $G$. As  $G$ is connected, there is a path $P_1$ in $G$ which start
at $u$, none of its  internal vertices is on $C$,  and $\upsilon(P_1)\in V(C)$.  Fix  $w\in N_G(\upsilon(P_1))\cap V(C)$ and let  $P_2$ be  the  path on $C$ between $\upsilon(P_1)$ and $w$ whose  length is more that $1$. If $P$ is  the path  between  $u$ and $w$   formed by  $P_1$ and $P_2$,        then   it is clear that  $d_{T(G, u)}(P)<d_G(\upsilon(P))$. This contradiction completes  the proof.
\end{proof}

In the following consequence, we give  some  lower   bounds on  the largest zero of the Laplacian matching polynomial.

\begin{corollary}
Let $G$ be a connected graph. Then
$$\lambda\big(\mathscr{L\hspace{-0.7mm}M}(G, x)\big)\geq\max\left\{\mathnormal{\Delta}(G)+1, \delta(G)+\sqrt{\mathnormal{\Delta}(G)}\right\}$$
with the   equality holds if and only  if   $G$ is a  star.
\end{corollary}

\begin{proof}
Let    $u\in V(G)$ be of   degree $\mathnormal{\Delta}(G)$. Indeed,   $d_{T(G, u)}(u)=d_G(u)$ and  therefore
$\mathnormal{\Delta}(T(G, u))=\mathnormal{\Delta}(G)$.
For each  connected graph $\mathnormal{\Gamma}$,    Proposition 3.9.3  of    \cite{Brouwer} states  that
$\lambda(L(\mathnormal{\Gamma}))\geq\mathnormal{\Delta}(\mathnormal{\Gamma})+1$
with the equality holds if and only if $\mathnormal{\Delta}(\mathnormal{\Gamma})=|V(\mathnormal{\Gamma})|-1$.
By this fact and     Corollary \ref{geqsignless}, we obtain that
$\lambda(\mathscr{L\hspace{-0.7mm}M}(G,x))\geq\lambda(L(T(G, u)))\geq\mathnormal{\Delta}(T(G, u))+1=\mathnormal{\Delta}(G)+1$, and moreover,
the equality $\lambda(\mathscr{L\hspace{-0.7mm}M}(G,x))=\mathnormal{\Delta}(G)+1$ holds  if and only if $G$ is a star.

For each  connected graph $\mathnormal{\Gamma}$,     the Perron--Frobenius theorem  \cite[Theorem 2.2.1]{Brouwer} implies that
$\lambda(A(\mathnormal{\Gamma}))\geq\sqrt{\mathnormal{\Delta}(\mathnormal{\Gamma})}$ with the equality holds if and only if $\mathnormal{\Gamma}                                          $ is a star.
Using this fact, Theorem  \ref{largestzero},  and  the Weyl   inequality  \cite[Theorem 2.8.1]{Brouwer}, we  derive
\begin{align}\label{nostar}\nonumber\lambda\big(\mathscr{L\hspace{-0.7mm}M}(G, x)\big)&=\lambda\Big(D_G\big(T(G, u)\big)+A\big(T(G, u)\big)\Big)\\\nonumber
&\geq\delta(G)+\lambda\Big(A\big(T(G, u)\big)\Big)\\
&\geq\delta(G)+\sqrt{\mathnormal{\Delta}\big(T(G, u)\big)}\\\nonumber
&=\delta(G)+\sqrt{\mathnormal{\Delta}(G)}.
\end{align}
Suppose that  the equality $\lambda(\mathscr{L\hspace{-0.7mm}M}(G, x))=\delta(G)+\sqrt{\mathnormal{\Delta}(G)}$ holds. So, the equality  in      \eqref{nostar}  is attained,  and thus, $T(G, u)$ is a star. This implies that $G$ is a star, and then,    $\lambda(\mathscr{L\hspace{-0.7mm}M}(G, x))=\delta(G)+\sqrt{\mathnormal{\Delta}(G)}$ forces that $|V(G)|\leq2$. Since  the equality $\lambda(\mathscr{L\hspace{-0.7mm}M}(G, x))=\delta(G)+\sqrt{\mathnormal{\Delta}(G)}$ is valid  for the stars   $G$ on at most $2$  vertices, the proof is complete.
\end{proof}

In the  following   theorem, we   establish     \eqref{Lmatchingthm1} which   slightly  improves
the second statement of Theorem 2.6 of  \cite{Ali}.

\begin{theorem}\label{deltacosx}
Let  $G$ be a  connected    graph with  $\mathnormal{\Delta}(G)\geq2$ and let  $\ell(G)$ be   the length of the longest path in $G$.
Then,
\begin{equation}\label{asserx}\lambda\big(\mathscr{L\hspace{-0.7mm}M}(G, x)\big)\leq\mathnormal{\Delta(G)}+2\sqrt{\mathnormal{\Delta(G)}-1}\cos\frac{\pi}{2\ell(G)+2}\end{equation}
with the   equality holds if and only  if   $G$ is a    cycle.
\end{theorem}

\begin{proof}
For simplicity, let $\mathnormal{\Delta}=\mathnormal{\Delta}(G)$ and $\ell=\ell(G)$.
For every positive integers $d$ and  $k\geq2$, the  {\it Bethe tree $B_{d, k}$} is a rooted  tree  with  $k$ levels in which  the    root vertex  is of    degree $d$, the vertices on  levels $2, \ldots,  k-1$  are of  degree $d+1$,   and the vertices on  level $k$ are of degree $1$.
By    Theorem 7 of \cite{Rojo},
\begin{equation}\label{uppereqq}
\lambda\big(A(B_{d, k})\big)=2\sqrt{d}\cos\frac{\pi}{k+1}.
\end{equation}
Let  $u\in V(G)$.
It is not hard to check that $T(G, u)$ is isomorphic to a  subgraph of $B_{\mathnormal{\Delta}-1, 2\ell+1}$.
For this, it is enough to correspond $u\in V(T(G, u))$ to  an arbitrary   vertex on  level $\ell+1$ in $B_{\mathnormal{\Delta}-1, 2\ell+1}$.
By applying  Theorem \ref{largestzero},   the Weyl   inequality  \cite[Theorem 2.8.1]{Brouwer},  the interlacing theorem \cite[Corollary 2.5.2]{Brouwer},   and \eqref{uppereqq}, we derive
\begin{align}\label{31decx}\nonumber\lambda\big(\mathscr{L\hspace{-0.7mm}M}(G, x)\big)&=\lambda\Big(D_G\big(T(G, u)\big)+A\big(T(G, u)\big)\Big)\\&\nonumber
\leq\lambda\Big(D_G\big(T(G, u)\big)\Big)+\lambda\Big(A\big(T(G, u)\big)\Big)\\&
\leq\mathnormal{\Delta}+\lambda\big(A(B_{\mathnormal{\Delta}-1,2\ell+1})\big)\\&\nonumber
=\mathnormal{\Delta}+2\sqrt{\mathnormal{\Delta}-1}\cos\frac{\pi}{2\ell+2},\end{align}
proving \eqref{asserx}.
Now, assume that the equality in \eqref{asserx} is achieved. Therefore,  the equality in \eqref{31decx} occurs,   and thus,
the Perron--Frobenius theorem  \cite[Theorem 2.2.1]{Brouwer} implies that   $T(G, u)$ is isomorphic to $B_{\mathnormal{\Delta}-1, 2\ell+1}$. Since  $\mathnormal{\Delta}\geq2$,  one   can easily  obtain   that   $G$ is a cycle.
Conversely, if $G$ is a cycle, then $T(G, u)$ is a  path   on $2\ell+1$ vertices.
By Theorem \ref{largestzero} and \eqref{uppereqq}, we get
$$\lambda\big(\mathscr{L\hspace{-0.7mm}M}(G, x)\big)=
\lambda\Big(D_G\big(T(G, u)\big)+A\big(T(G, u)\big)\Big)
=2+\lambda\big(A(B_{1, 2\ell+1})\big)
=2+2\cos\frac{\pi}{2\ell+2}.$$
This   completes the proof.
\end{proof}

Stevanovi\'{c} \cite{D} proved  that the  eigenvalues  of the adjacency matrix  of a tree $T$ are  less than $2\sqrt{\mathnormal{\Delta}(T)-1}$.
The corollary    below   gives an improvement of this upper bound  for  the  subdivision of trees.

\begin{corollary}
Let $G$ be a   graph with  $\mathnormal{\Delta}(G)\geq2$. Then
\begin{equation}\label{imx}\lambda\Big(\mathscr{M}\big(S(G), x)\big)\Big)<1+\sqrt{\mathnormal{\Delta}(G)-1}.\end{equation}
In particular, if $F$ is a forest  with  $\mathnormal{\Delta}(F)\geq2$, then $\lambda(A(S(F)))<1+\sqrt{\mathnormal{\Delta}(F)-1}$.
\end{corollary}

\begin{proof}
It follows from  Theorem \ref{deltacosx} that    $\lambda(\mathscr{L\hspace{-0.7mm}M}(G, x))<\mathnormal{\Delta}(G)+2\sqrt{\mathnormal{\Delta}(G)-1}$.
Moreover,  it   follows from  Corollary   \ref{XrootrelationX}  that   $\lambda(\mathscr{M}(S(G), x))=\sqrt{\lambda(\mathscr{L\hspace{-0.7mm}M}(G, x))}$. From these, we find that    $$\lambda(\mathscr{M}(S(G), x))<\sqrt{\mathnormal{\Delta}(G)+2\sqrt{\mathnormal{\Delta}(G)-1}}=1+\sqrt{\mathnormal{\Delta}(G)-1},$$
proving   \eqref{imx}.
As   the subdivision of a  forest   is    a forest, the `in particular' statement     follows     from Theorem \ref{treeequal} and  \eqref{imx}.
\end{proof}

\begin{remark}
Note that  $\mathnormal{\Delta}(S(G))=\mathnormal{\Delta}(G)$ for every   graph $G$ with $\mathnormal{\Delta}(G)\geq2$. So,
for the subdivision of a  graph with the maximum degree at least $2$, the upper bound which   appears  in  \eqref{imx} is sharper  than the    upper bound   that  comes from    \eqref{thmmatching1}.
\end{remark}

We demonstrated   in Theorem  \ref{largestzero}  that the  largest zero   of    the Laplacian matching polynomial  has the  multiplicity   $1$.
In the following theorem, we   prove the remaining statements of  \eqref{Lmatchingthm3} as   analogues   of   the results given in \eqref{thmmatching2}.

\begin{theorem}\label{multi}
Let  $G$ be a     graph and let  $n=|V(G)|$.
For each edge $e\in E(G)$,  the zeros of  $\mathscr{L\hspace{-0.7mm}M}(G,x)$ and $\mathscr{L\hspace{-0.7mm}M}(G-e,x)$  interlace  in the sense that,  if     $\alpha_1\leq\cdots\leq\alpha_n$ and
$\beta_1\leq\cdots\leq\beta_n$ are respectively   the zeros  of  $\mathscr{L\hspace{-0.7mm}M}(G,x)$ and $\mathscr{L\hspace{-0.7mm}M}(G-e,x)$,  then $\beta_1\leq\alpha_1\leq\beta_2\leq\alpha_2\leq\cdots\leq\beta_n\leq\alpha_n$. Also, if $G$ is connected, then   $\lambda(\mathscr{L\hspace{-0.7mm}M}(G, x))>\lambda(\mathscr{L\hspace{-0.7mm}M}(H, x))$  for any  proper subgraph $H$ of $G$.
\end{theorem}

\begin{proof}
Fix  an edge    $e\in E(G)$ and  denote by    $\upsilon_e$     the vertex of $S(G)$  corresponding to   $e$.
Let  $\alpha_1\leq\cdots\leq\alpha_n$ and
$\beta_1\leq\cdots\leq\beta_n$ be    the zeros  of  $\mathscr{L\hspace{-0.7mm}M}(G,x)$ and $\mathscr{L\hspace{-0.7mm}M}(G-e,x)$, respectively.
Corollary \ref{XrootrelationX} yields that
$\sqrt{\alpha_1}\leq\cdots\leq\sqrt{\alpha_n}$ is the  end part  of   the  nondescending sequence which  consists  of all the  zeros of $\mathscr{M}(S(G), x)$ and
$\sqrt{\beta_1}\leq\cdots\leq\sqrt{\beta_n}$ is the  end part  of   the  nondescending sequence  which  consists of all the  zeros of $\mathscr{M}(S(G-e), x)$.
As $S(G-e)=S(G)-\upsilon_e$, it follows from   \eqref{thmmatching2} that    the zeros of $\mathscr{M}(S(G), x)$ and $\mathscr{M}(S(G-e), x)$  interlace. So, we find that
$$\sqrt{\beta_1}\leq\sqrt{\alpha_1}\leq\sqrt{\beta_2}\leq\sqrt{\alpha_2}\leq\cdots\leq\sqrt{\beta_n}\leq\sqrt{\alpha_n}$$ which means that
$\beta_1\leq\alpha_1\leq\beta_2\leq\alpha_2\leq\cdots\leq\beta_n\leq\alpha_n$, as desired.

Now, assume that $G$ is connected.
Let $H$ be a proper subgraph of $G$ and  let   $u\in V(H)$.
As     $T(H, u)$ is a proper subgraph of   $T(G, u)$,
if  $R$ denotes   the  submatrix of $D_G(T(G, u))+A(T(G, u))$ corresponding to  the  vertices  in $V(T(H, u))$, then
$R-(D_H(T(H, u))+A(T(H, u)))$ is  a nonzero matrix with   nonnegative entries.
So,  by   applying    Theorem \ref{largestzero} and the Perron--Frobenius theorem  \cite[Theorem 2.2.1]{Brouwer},    we get
\begin{align*}\lambda\big(\mathscr{L\hspace{-0.7mm}M}(G, x)\big)&=\lambda\Big(D_G\big(T(G, u)\big)+A\big(T(G, u)\big)\Big)\\&>\lambda(R)\\&>\lambda\Big(D_H\big(T(H, u)\big)+A\big(T(H, u)\big)\Big)\\&=\lambda\big(\mathscr{L\hspace{-0.7mm}M}(H, x)\big).\qedhere\end{align*}
\end{proof}

\begin{remark}\label{ghor}
For every graph $G$ and  real number $\alpha$, let  $m_G(\alpha)$ denote  the multiplicity of $\alpha$  as a root of $\mathscr{L\hspace{-0.7mm}M}(G, x)$.
As a consequence of  Theorem  \ref{multi}, we have  $|m_G(\alpha)-m_{G-e}(\alpha)|\leq1$ for each edge  $e\in E(G)$.
\end{remark}

It is known  that among all trees with a fixed number of vertices the path  has
the smallest value of the largest  Laplacian eigenvalue \cite{pet}.
The following result
can be considered as an analogue of this  fact and
is  obtained from    Theorems \ref{treeequal} and \ref{multi}.

\begin{corollary}
Let $P_n$ and $K_n$ be the path  and complete graph  on $n$ vertices, respectively. For  any
connected graph $G$  on $n$ vertices which is not  $P_n$ and  $K_n$,
$$\lambda\big(\mathscr{L\hspace{-0.7mm}M}(P_n, x)\big)<\lambda\big(\mathscr{L\hspace{-0.7mm}M}(G, x)\big)<\lambda\big(\mathscr{L\hspace{-0.7mm}M}(K_n, x)\big).$$
\end{corollary}

\section{Concluding remarks}

In this paper, we have discovered some properties of the location of zeros of the Laplacian matching polynomial. Most of our results can  be considered as  analogues of   known  results on the matching polynomial.
Comparing to the matching polynomial, the Laplacian matching polynomial contains not only the information of the sizes of   matchings in  the graph
but also the  vertex degrees    of the graph. Hence, it seems to be that more  structural properties
of  graphs  can be reflected by  the Laplacian matching polynomial rather than the
matching polynomial. For an  instance,   $0$ is a root of $\mathscr{L\hspace{-0.7mm}M}(G, x)$  if and only if $G$ is a
forest, in while  $0$ is  a root of $\mathscr{M}(G, x)$ if and only if $G$ has no prefect matchings.

More interesting facts about the Laplacian matching polynomial can be concerned in further. For example,  one may focus on    the   multiplicities  of   zeros  of the  Laplacian    matching polynomial  as    there are  many results on  the multiplicities  of zeros  of the   matching polynomial. In view of Remark \ref{ghor}, for every graph $G$ and  real number $\alpha$, one may divide   $E(G)$  into three subsets  based on how   the multiplicity of   $\alpha$  changes when an edge of $G$   is removed. The      corresponding problem   about the matching polynomial is  investigated  by   Chen  and Ku    \cite{Ku}. Also,  it is  a known result   that the   multiplicity of a zero  of the   matching polynomial  is at most the   path partition number  of the graph,  that is, the minimum number of vertex disjoint paths required to cover all the vertices of  the graph \cite[Theorem 6.4.5]{Godsil}.   It seems to be an interesting problem to find a sharp    upper bound   on  the   multiplicity of  a zero  of the  Laplacian    matching polynomial.

\end{document}